\documentclass[12pt]{article}
\usepackage{amsmath,amsfonts,amsthm,amscd,upref,amstext}
\usepackage[dvips]{graphicx}
\usepackage{subfigure}
\usepackage{pb-diagram,pb-xy}
\usepackage[all]{xy}

\newtheorem{prop}{Proposition}[section]
\newtheorem{thm}[prop]{Theorem}

\newtheorem{cor}[prop]{Corollary}

\newtheorem{defi}[prop]{Definition}
\newtheorem{rem}[prop]{Remark}
\newtheorem{lem}[prop]{Lemma}
\newtheorem{nota}[prop]{Notation}
\newtheorem*{agra}{Acknowlegment}

\numberwithin{equation}{section}

\begin{document}

\title{Equimultiple Coefficient Ideals}
\author{P. H. Lima
\thanks{Work partially supported by CAPES-Brazil 10056/12-2.}\,\,\,\, and \,\,\,\,V. H. Jorge P\'erez
\thanks{Work partially supported by FAPESP-Brazil 2012/20304-1 and CNPq-Brazil 303682/2012-4. {\it Key words}: coefficient ideals, integral closure, reduction, multiplicity}}

\date{}
\maketitle

\begin{abstract}
Let $(R,\mathfrak{m})$ be a quasi-unmixed local ring and $I$ an equimultiple ideal of $R$ of analytic spread $s$.
In this paper, we introduce the equimultiple coefficient ideals. Fix $k\in \{1,...,s\}.$
The largest ideal $L$ containing $I$ such that
$e_{i}(I_{\mathfrak{p}})=e_{i}(L_{\mathfrak{p}})$ for each $i \in \{1,...,k\}$
and each minimal prime $\mathfrak{p}$ of $I$ is called the $k$-th equimultiple coefficient ideal denoted by  $I_{k}$.
It is a generalization of the coefficient ideals firstly introduced by Shah \cite{S} for the case of $\mathfrak{m}$-primary ideals.
We also see applications of these ideals. For instance, we show that the associated graded ring $G_{I}(R)$ satisfies
the $S_{1}$ condition if and only if $I^{n}=(I^{n})_{1}$ for all $n$.

\end{abstract}

\section{Introduction}
Let $(R,\mathfrak{m})$ be a quasi-unmixed local ring of dimension $d$ and $I$ an $\mathfrak{m}$-primary ideal of $R$.
Shah \cite{S} showed the existence of unique largest ideals $I_{k}$ ($1\leq k \leq d$) lying between $I$ and $\overline{I}$
such that the $k+1$ Hilbert coefficients of $I$ and $I_{k}$ coincide, that is, $e_{i}(I) = e_{i}(I_{k})$ for $0\leq i \leq k$.
These ideals are called coefficient ideals. They have been studying in some articles such as
\cite{C}, \cite{HJLS}, \cite{HLS}, \cite{S}.
In \cite{S}, it is found that if $I$ contains a regular element, then the Ratliff-Rush closure $I^{*}$ and the $d$-th coefficient ideal $I_{d}$ coincide;
moreover the author studied the associated primes of the associated graded ring $G_{I}(R)$.
In \cite{C}, Ciuperc$\breve{a}$ studies the relation between the $S_{2}$-ification of the
extended Rees algebra $S = R[It,t^{-1}]$ and the cited ideals.
In \cite{HLS}, when $R$ is a domain, it is shown that the
associated Ratliff-Rush ideal $I^{*}$ of $I$ is the contraction to $R$ of the extension
of $I$ to its blowup $\mathcal{B}(I) =\{R[I/a]_{P} \ | \ a\in I-0, \ P\in \text{spec}(R[I/a]) \}$, i.e,
$I^{*}=\bigcap \{ IS \cap R \ | \ S \in \mathcal{B}(I) \}$. If further $R$ is analytically unramified, it is shown in \cite{HJLS},
that the coefficient ideals $I_{k}$ are also contracted from a blowup $\mathcal{B}(I)^{(k)}$ which is obtained from $\mathcal{B}(I)$
by a process similar to ``$S_{2}$-ification''.

The paper is organized as follows: In section 2, it is generalized the notion of coefficient ideals (introduced by Shah); we work with an equimultiple ideal $I$, that is,
$\text{ht}(I)=s(I)$, where $s=s(I)$ is the analytic spread of $I$.
We have made use of the B\"{o}ger theorem (on Hilbert-Samuel multiplicity)
to show the existence of unique largest ideals $I_{k}$ (we use the same notation as used by Shah)
lying between $I$ and $\overline{I}$,
and satisfying
$e_{i}(I_{\mathfrak{p}})=e_{i}((I_{k})_{\mathfrak{p}})$ for $0 \leq i \leq s$ and every minimal prime $\mathfrak{p}$ of $I$.
We call them equimultiple coefficient ideals. Given an ideal $J$, we denote the unmixed part of $J$ by $J^{u}$. We show that if $I$ contains a regular element then $I_{s}=(I^{*})^{u}$, which shows $I_{s}$ is an unmixed ideal. Actually we verify that
all the equimultiple coefficient ideals are unmixed (Theorem \ref{coef_unmixed}).

In section 3, it is given a criterion to control the height of the associated primes of $G_{I}(R)$ (Theorem \ref{height_ass}).
As a consequence of this, we show that if $G_{I}(R)$ satisfies $S_{1}$, so does $G_{I^{m}}(R)$ for every $m$ (Corollary \ref{S1property}).
In \cite{NV}, Noh and Vasconcelos showed that if $R$ is a Cohen-Macaulay ring, the Rees algebra $R[It]$ satisfies $S_{2}$ and $I$ is an equimultiple ideal,
then all the powers $I^{n}$ are unmixed ideals. In this work, we verify the same result when $R$ is only a quasi-unmixed ring satisfying the $S_{2}$ condition.
Finally, we give a way to provide associated graded rings $G_{I}(R)$ satisfying the $S_{1}$ condition
(Corollary \ref{construct_of_G_I(R)}).

\section{Equimultiple Coefficient Ideals}
In this section, we show the existence of the equimultiple coefficient ideals and also we introduce a refined version for their existence.
We show that all of them are unmixed ideals, and
find their primary decompositions components. It is also seen how coefficient ideals control the height of the associated primes of $G_{I}(R)$.
For example, the associated graded ring $G_{I}(R)$ satisfies the $S_{1}$ condition
if and only if $(I^{n})_{1}=I^{n}$ for all $n$. As consequence, if $G_{I}(R)$ satisfies the $S_{1}$ condition then
$G_{I^{m}}(R)$ satisfies the $S_{1}$ condition for all $m$. Finally, we give a way to construct associated graded rings satisfying the $S_{1}$ condition.

We begin recalling the well known B\"{o}ger's theorem on Hilbert coefficients.
\begin{thm} \emph{(B\"{o}ger)}
Let $(R,\mathfrak{m})$ be a quasi-unmixed
local ring, and let $I �\subseteq J$ be two ideals such that $I$ is equimultiple. Then $J �\subseteq \overline{I}$ if
and only if $e_{0}(I_{\mathfrak{p}}) = e_{0}(J_{\mathfrak{p}})$ for every $\mathfrak{p}\in \emph{Min}(R/I)$.
\end{thm}

The next two remarks may be found in \cite{S}. They are used when we localize an equimultiple ideal $I$ at a minimal prime $P$.

\begin{rem}\label{rem1}
Let $(R,\mathfrak{m})$ be a Noetherian local ring and $\dim R\geq 1$. Suppose $I \subseteq J$ are $\mathfrak{m}$-primary ideals and fix $k$ such that $1\leq k \leq d$.
Then for all large $n$, $e_{i}(I)=e_{i}(J)$ with $0\leq i \leq k$ if and only if $\ell(J^{n}/I^{n})\leq P(n)$, where $P(n)$ is some polynomial in $n$ of degree at most $d-(k+1)$.
\end{rem}

\begin{proof}
It suffices to observe that, for large $n$,
$$
\ell(J^{n}/I^{n})=\ell(R/I^{n})-\ell(R/J^{n})=\sum_{i=0}^{i=d}(-1)^{i}[e_{i}(I)-e_{i}(J)]\bigg(\begin{array}{c}
                                                                                           n+d-i-1 \\
                                                                                           d-i
                                                                                         \end{array}\bigg)
.$$
\end{proof}

\begin{rem}\label{rem2}
Let $(R,\mathfrak{m})$ be a Noetherian local ring and $\dim R\geq 1$. Suppose $I \subseteq I' \subseteq J$ are $\mathfrak{m}$-primary ideals and fix $k$ such that $1\leq k \leq d$. Then $e_{i}(I)=e_{i}(J)$ with $0\leq i \leq k$ if and only if $e_{i}(I)=e_{i}(I')=e_{i}(J)$ with $0\leq i \leq k$.
\end{rem}

\begin{proof}
We just use that $\ell(I'^{n}/I^{n})\leq \ell(J^{n}/I^{n})$ and apply Remark \ref{rem1}.
\end{proof}

By B\"{o}ger's Theorem it is easy to see that $\overline{I}$ is the unique largest ideal $L$ which satisfies
$e_{0}(I_{\mathfrak{p}})=e_{0}(L_{\mathfrak{p}})$ for every $ \mathfrak{p} \ \in \text{Min}(R/I)$. In the next result we generalize the notion of coefficient ideals, firstly introduced by Shah in \cite{S}, for a more general case which $I$ is an equimultiple ideal.

\begin{thm} \label{shah'thm_for_equi}\emph{(Existence of the equimultiple coefficients ideals)}
Let $(R,\mathfrak{m})$ be a quasi-unmixed local ring. Assume $R/\mathfrak{m}$ is infinite and $\dim R = d\geq 1$.  Let $I$ be an equimultiple ideal.
Then there exist unique largest ideals $I_{k}$, for $1\leq k\leq s$, containing $I$ such that
\begin{itemize}
  \item [\emph{(1)}] $e_{i}(I_{\mathfrak{p}})=e_{i}((I_{k})_{\mathfrak{p}})$ \ for \ $0 \leq i \leq s \ and  \ every \ \mathfrak{p} \ \in \emph{Min}(R/I)$;
  \item [\emph{(2)}] $I \subseteq I_{s} \subseteq \cdot \cdot \cdot \subseteq I_{1} \subseteq \overline{I}$.
\end{itemize}
\end{thm}

\begin{proof}
Let $s=s(I)$ denote the analytic spread of $I$. By a known Ratliff-Rush theorem we have $s= \dim R_{\mathfrak{p}}$ for any $\mathfrak{p}\in \text{Min}(R/I)$.
For each $k=1,...,s$, consider the set
$$
\begin{array}{c}
  \hspace{-3cm}V_{k}=\{ L \ | \ L \ \text{is an ideal of} \ R \ \text{such that} \ L \supseteq I \ \text{and} \\
   \hspace{2cm}  e_{i}(I_{\mathfrak{p}})=e_{i}(L_{\mathfrak{p}}) \ \text{for} \ \text{every} \ 0 \leq i \leq k \ \text{and} \ \mathfrak{p}\in \text{Min}(R/I) \}
\end{array}
$$
Firstly note that if $L\in V_{k}$ then $e_{i}(I_{\mathfrak{p}})=e_{i}(L_{\mathfrak{p}})$ for every $\mathfrak{p}\in \text{Min}(R/I)$ and in particular, by B\"{o}ger's Theorem, $L\subseteq \overline{I}$.

Since $I\in V_{k}$ and $R$ is Noetherian there exists a maximal element $J \in V_{k}$. We prove $J$ is unique. Let $L \in V_{k}$ and $x\in L$. Since
$I \subseteq (I,x) \subseteq L$, we have by Remark \ref{rem2} that $e_{i}(I_{\mathfrak{p}})=e_{i}((I,x)_{\mathfrak{p}})=e_{i}(J_{\mathfrak{p}})$ for $0\leq i \leq k$ and
$\mathfrak{p} \in \text{Min}(R/I)$. Then $I$ is a reduction of $(I,x)$ so that $(I,x)^{t+1}=(I,x)^{t}I$ for some $t$. So $x^{t+1}\in (I,x)^{t}I\subseteq (J,x)^{t}J$.
Hence, $(J,x)^{t+1}=(J,x)^{t}J$ and then $(J,x)^{n}=(J,x)^{t}J^{n-t}$ for $n\geq t$. Fix $\mathfrak{p} \in \text{Min}(R/I)$. We have
$$
\begin{array}{cl}
 \hspace{-0.2cm} \ell( (J,x)^{n}_{\mathfrak{p}}/ J^{n}_{\mathfrak{p}} ) &\hspace{-0.1cm} =\ell\big( ((J,x)^{t}J^{n-t})_{\mathfrak{p}}/ J^{n}_{\mathfrak{p}} \big)=\ell\big( (J^{n}_{\mathfrak{p}}, \ (J^{n-1}x)_{\mathfrak{p}},..., (J^{n-t}x^{t})_{\mathfrak{p}}) / J^{n}_{\mathfrak{p}} \big) \vspace{0.2cm}\\
    & \hspace{-0.1cm} \leq \sum_{i=1}^{t} \ell \big( (J^{n-i}x^{i})_{\mathfrak{p}}+J^{n}_{\mathfrak{p}}/ J^{n}_{\mathfrak{p}} \big) \vspace{0.2cm} \leq \sum_{i=1}^{t} \ell \big( (J^{n-i}x^{i})_{\mathfrak{p}}+J^{n}_{\mathfrak{p}}/ I^{n}_{\mathfrak{p}} \big) \vspace{0.2cm} \\
    & \hspace{-0.1cm}  \leq \sum_{i=1}^{t} \left[ \ell \big( (J^{n-i}x^{i})_{\mathfrak{p}}+I^{n}_{\mathfrak{p}}/ I^{n}_{\mathfrak{p}} \big) + \ell(J^{n}_{\mathfrak{p}}/I^{n}_{\mathfrak{p}}) \right] \vspace{0.2cm}\\
    & \hspace{-0.1cm} \leq \sum_{i=1}^{t} \left[ \ell \big( (I^{n-i}x^{i})_{\mathfrak{p}}+I^{n}_{\mathfrak{p}}/ I^{n}_{\mathfrak{p}} \big) + \ell \bigg( \frac{(J^{n-i}x^{i})_{\mathfrak{p}}+I^{n}_{\mathfrak{p}} }{(I^{n-i}x^{i})_{\mathfrak{p}}+ I^{n}_{\mathfrak{p}}} \bigg) + \ell(J^{n}_{\mathfrak{p}}/I^{n}_{\mathfrak{p}}) \right] \vspace{0.2cm}\\
    & \hspace{-0.1cm}  \leq \sum_{i=1}^{t} \left[ \ell(J^{n-i}_{\mathfrak{p}}/I^{n-i}_{\mathfrak{p}})+\ell((I,x)^{n}_{\mathfrak{p}}/I^{n}_{\mathfrak{p}})+\ell(J^{n}_{\mathfrak{p}}/I^{n}_{\mathfrak{p}}) \right].
\end{array}
$$
Since $e_{i}(I_{\mathfrak{p}})=e_{i}((I,x)_{\mathfrak{p}})$ and $e_{i}(I_{\mathfrak{p}})=e_{i}(J_{\mathfrak{p}})$ holds for $0\leq i \leq k$ and every $\mathfrak{p} \in \text{Min}(R/I)$ one can conclude by Remark \ref{rem1} that $e_{i}(J_{\mathfrak{p}})=e_{i}((J,x)_{\mathfrak{p}})$ holds for $0\leq i \leq k$ and every $\mathfrak{p} \in \text{Min}(R/I)$. But $J$ is maximal in $V_{k}$, so $L\subseteq J$ and
therefore $J$ is the unique maximal in $V_{k}$. This ideal is denoted by $I_{k}$.
\end{proof}

\begin{defi}
The ideals $I_{k}$ above obtained will be called \emph{equimultiple coefficient ideals}.
\end{defi}

In order to introduce a refined version of the above theorem we define the following.

\begin{defi}
Let $(R,\mathfrak{m})$ be a local ring and $I$ an ideal of $R$. Let $I^{sat}=\bigcup_{n\geq 1}(I: \mathfrak{m}^{n})$ be the saturation of $I$. Then
the ideal $\mathfrak{q}(I)=\overline{I}\cap I^{sat}$ is called of relative integral closure of $I$.
\end{defi}

Below we just introduce a refined version for the above theorem. By \cite[Proposition 4.3]{HPV} and B\"{o}ger's Theorem is also
easy to see that $\mathfrak{q}(I)$ is the unique largest ideal $L$ which satisfies
$e_{0}(I_{\mathfrak{p}})=e_{0}(L_{\mathfrak{p}})$ for every $ \mathfrak{p} \ \in \text{Min}(R/I)$.

\begin{thm} \emph{(Refined equimultiple coefficients ideals)}
Let $(R,\mathfrak{m})$ be a quasi-unmixed local ring. Assume that $R/\mathfrak{m}$ is infinite and $\dim R = d\geq 1$. Let $I$ be an equimultiple ideal. Then there exist unique largest ideals $I_{k}'$ with $\ell(I_{k}'/I)<\infty$, for $1\leq k\leq s$, containing $I$ such that
\begin{itemize}
  \item [\emph{(1)}] $e_{i}(I_{\mathfrak{p}})=e_{i}((I_{k}')_{\mathfrak{p}})$ \ for \ $0 \leq i \leq k \ and  \ every \ \mathfrak{p} \ \in \emph{Min}(R/I)$;
  \item [\emph{(2)}] $I \subseteq I_{s}' \subseteq \cdot \cdot \cdot \subseteq I_{1}' \subseteq \mathfrak{q}(I)$.
\end{itemize}
\end{thm}
\begin{proof}
It suffices to require that the elements $L$ in the above sets $V_{k}$ are also such that $\ell(L/I)<\infty$.
\end{proof}

Let $I^{*}=\bigcup_{n\geq 1}(I^{n+1}:I^{n})$ be the Ratliff-Rush ideal. It is known that $I^{*}$ is the largest ideal $L$ which satisfies $L^{n}=I^{n}$ for large $n$.
Moreover, by localizing at each $\mathfrak{p} \in \text{Min} (R/I)$ we have $e_{i}(I_{\mathfrak{p}})=e_{i}((I^{*})_{\mathfrak{p}})$
for $1\leq i\leq s$ and every $\mathfrak{p}$ minimal prime of $I$.
By the above theorem one has $I^{*} \subseteq I_{s}$.

An ideal $I$ is said to be a Ratliff-Rush ideal if $I^{*}=I$.
\begin{cor}\label{cor_equi}
Assume the hypothesis of Theorem \ref{shah'thm_for_equi}. Then
$$
\begin{array}{rl}
  I \subseteq J \subseteq I_{k} \subseteq \overline{I}\hspace{-0.2cm} & \text{ iff } \hspace{0.2cm}I \subseteq J, \ \text{and }  e_{i}(I_{\mathfrak{p}})=e_{i}(J_{\mathfrak{p}}) \text{ for } 1\leq i\leq s \\
   & \hspace{0.8cm}\text{ and every } \mathfrak{p} \in \emph{Min}(R/I).
\end{array}
$$
\end{cor}

\begin{cor}\label{rat}
Assume the hypothesis of Theorem \ref{shah'thm_for_equi}. All the coefficient ideals $I_{k}$ are Ratliff-Rush ideals.
\end{cor}

\begin{proof}
We know that $(I_{k})_{\mathfrak{p}}^{n}=((I_{k})^{*})_{\mathfrak{p}}^{n}$ for $n\gg 0$ and any prime $\mathfrak{p}$. In particular,
$e_{i}((I_{k})_{\mathfrak{p}})=e_{i}(((I_{k})^{*})_{\mathfrak{p}})$ for $0\leq i \leq s$ and any $\mathfrak{p}\in \text{Min}(R/I)$. So by maximality of $I_{k}$, we have $(I_{k})^{*}\subseteq I_{k}$.
\end{proof}

\begin{nota}
Given an ideal $J \subseteq R$, let $J^{u}$ denote the unmixed part of $J$.
\end{nota}

The next result shows that the coefficient ideal $I_{s}$ is an unmixed ideal if $I$ contains a regular element. Later, we will see that in fact $I_{s}$
is unmixed anyway (see Theorem \ref{coef_unmixed}).

\begin{prop}
Assume the setup of Theorem \ref{shah'thm_for_equi} with $I$ containing a regular element. Then $I_{s}=(I^{*})^{u}$. In particular, $I_{s}$ is an unmixed ideal.
\end{prop}

\begin{proof}
By simplicity of notation, let $J=(I^{*})^{u}$ denote the unmixed part of the Ratliff-Rush closure $I^{*}$. We have $\text{Min}(R/I)=\text{Min}(R/I^{*})=\text{Min}(R/J)$.
Because of first condition in Theorem \ref{shah'thm_for_equi}, we have $(I_{s})_{\mathfrak{p}} \subseteq (I_{\mathfrak{p}})^{*}=I^{*}_{\mathfrak{p}}$ for every $\mathfrak{p}\in \text{Min}(R/I)$.
But $(I_{s})_{\mathfrak{p}} \subseteq I^{*}_{\mathfrak{p}}= J_{\mathfrak{p}}$ for any $\mathfrak{p} \in \text{Min}(R/J)=\text{Ass}(R/J)$. Therefore, $I_{s}\subseteq J$.

Now note that by a property of Ratliff-Rush closure, one has $I_{\mathfrak{p}}^{n}=(I_{\mathfrak{p}}^{*})^{n}=J_{\mathfrak{p}}^{n}$ for large $n$ and any $\mathfrak{p}\in \text{Min}(R/I)$. Thus, $e_{i}(I_{\mathfrak{p}})=e_{i}(J_{\mathfrak{p}})$ for
$0\leq i \leq s$ and any $\mathfrak{p}\in \text{Min}(R/I)$, so that $J \subseteq I_{s}$ by maximality of $I_{s}$.
\end{proof}

\begin{prop}\label{equality_k}
Assume the setup of Theorem \ref{shah'thm_for_equi} and let $J\supseteq I$ be an equimultiple ideal. Then
\begin{itemize}
  \item [\emph{(1)}] if $J\subseteq I_{k}$ then $I_{k}=J_{k}$.
  \item [\emph{(2)}] if there exists one positive integer $m$ such that $J^{m} \subseteq (I^{m})_{k}$
  then $J^{n} \subseteq (I^{n})_{k}$ for all positive integer $n$.
  \item [\emph{(3)}] $(((I^{m})_{k})^{n})_{k} = (I^{mn})_{k}$ for all positive integer $m,n$.
\end{itemize}

\end{prop}

\begin{proof}
Fix $k$. For item $\text{(1)}$, it suffices to use $\ell(R/I_{\mathfrak{p}}^{n})- \ell(R/J_{\mathfrak{p}}^{n})\leq \ell(R/I_{\mathfrak{p}}^{n})- \ell(R/(I_{k})_{\mathfrak{p}}^{n})$ for each prime $\mathfrak{p}\in \text{Min}(R/I)$ and the fact that the last term is, for large $n$, a polynomial of degree at most $s-(k+1)$.

Now we show $\text{(2)}$.
We have $e_{i}(I^{m}_{\mathfrak{p}})=e_{i}(J^{m}_{\mathfrak{p}})$ for $0\leq i \leq k$ and every $\mathfrak{p}\in \text{Min}(R/I^{m})$, by Corollary \ref{cor_equi}. By using coefficients ideals
for primary case, we have, for each minimal prime $\mathfrak{p}$, that
$I^{m}_{\mathfrak{p}} \subseteq J^{m}_{\mathfrak{p}} \subseteq (I^{m}_{\mathfrak{p}})_{k}$ for $0\leq i \leq k$.
By \cite[Proposition 3.2]{HJLS}, $I^{n}_{\mathfrak{p}} \subseteq J^{n}_{\mathfrak{p}} \subseteq (I^{n}_{\mathfrak{p}})_{k}$ for all $n$, so that
for each minimal prime $\mathfrak{p}$, we have $e_{i}(I^{n}_{\mathfrak{p}})=e_{i}(J^{n}_{\mathfrak{p}})$ for $0\leq i \leq k$ and all $n$. Corollary \ref{cor_equi}
gives then $J^{n}\subseteq (I^{n})_{k}$ for all $n$.

Item $\text{(3)}$ is a combination of the two previous items.
\end{proof}

\begin{rem}
If $I$ contains a regular element, by Proposition \ref{equality_k} and Corollary \ref{rat} we have $(I^{*})_{k}=I_{k}=(I_{k})^{*}$, since $(I^{*})_{k}=I_{k}$.
In particular, the unmixed part of a Ratliff-Rush ideal is also a Ratliff-Rush ideal.
\end{rem}

\begin{thm}\label{coef_unmixed}
Assume the setup of Theorem \ref{shah'thm_for_equi}. The coefficient ideals of $I^{u}$ are $(I^{u})_{k}=(I_{k})^{u}$, and all
the coefficient ideals of $I$ are unmixed ideals.
\end{thm}

\begin{proof}
Firstly we construct a specific chain from $I^{u}$ formed by the unmixed part of the coefficient ideals of $I$ and after show that its terms are the coefficient ideals of $I^{u}$.
We have $((I_{1})^{u})_{\mathfrak{p}}=(I_{1})_{\mathfrak{p}} \subseteq \overline{I}_{\mathfrak{p}}$ for any $\mathfrak{p}\in \text{Ass}(R/\overline{I})$. So $(I_{1})^{u}\subseteq \overline{I}$. Moreover,
$((I_{2})^{u})_{\mathfrak{p}}=(I_{2})_{\mathfrak{p}} \subseteq ((I_{1})^{u})_{\mathfrak{p}}$ for any $\mathfrak{p}\in \text{Ass}(R/(I_{1})^{u})$ so that $(I_{2})^{u}\subseteq (I_{1})^{u}$. Inductively the desired chain is constructed.

Now let $I^{u}\subseteq J_{s} \subseteq \cdot \cdot \cdot \subseteq J_{1} \subseteq \overline{I}$ be the coefficient ideals of $I^{u}$.
Fix $k$. So $J_{k}$ is the unique largest ideal for which $e_{i}((I^{u})_{\mathfrak{p}})=e_{i}((J_{k})_{\mathfrak{p}})$ for $0\leq i \leq k$ and any $\mathfrak{p} \in \text{Min}(R/I^{u})$. Moreover
$e_{i}((I^{u})_{\mathfrak{p}})=e_{i}(I_{\mathfrak{p}})=e_{i}((I_{k})_{\mathfrak{p}})=e_{i}(((I_{k})^{u})_{\mathfrak{p}})$ for $0\leq i \leq k$ and any $\mathfrak{p} \in \text{Min}(R/I^{u})$.
Hence, $J_{k}\subseteq I_{k}\subseteq (I_{k})^{u}$ and therefore $(I_{k})^{u}=J_{k}$ for each $k$.

By Proposition \ref{equality_k}, $(I^{u})_{k}= I_{k}$. Therefore, $I_{k}$ is an unmixed ideal for each $k$.
\end{proof}

Next result expresses the primary decomposition components of the coefficient ideals $I_{k}$, besides giving another way to show they are unmixed ideals.

\begin{prop}\label{prym_deco}
Assume the setup of Theorem \ref{shah'thm_for_equi}. Let $\mathfrak{p}_{1},...,\mathfrak{p}_{r}$ be the minimal primes of $I$. Then
$I_{k}$ has the following primary decomposition
$$
I_{k}=((I_{\mathfrak{p}_{1}})_{k} \cap R)\cap \cdot \cdot \cdot \cap ((I_{\mathfrak{p}_{r}})_{k} \cap R).
$$
Furthermore, $(I_{k})_{\mathfrak{p}}=(I_{\mathfrak{p}})_{k}$ for every prime ideal $\mathfrak{p}$.
\end{prop}

\begin{proof}
Fix $k \in \{1,...,s\}$. To simplify notations let $J_{i}$, $H_{i}$ denote $(I_{\mathfrak{p}_{i}})_{k}$ and $(I_{\mathfrak{p}_{i}})_{k} \cap R$ respectively, where
$1\leq i \leq r$.
Since $J_{i}$ is $\mathfrak{p}_{i}R_{\mathfrak{p}_{i}}$-primary, $H_{i}$ is $\mathfrak{p}_{i}$-primary.
Set $H=H_{1}\cap \cdot \cdot \cdot \cap H_{r}$. Then $H_{\mathfrak{p}_{i}}=(H_{i})_{\mathfrak{p}_{i}}=J_{i}$ for each $i$, as
$(H_{j})R_{\mathfrak{p}_{i}}=R_{\mathfrak{p}_{i}}$ for every $j \neq i$. By definition of equimultiple coefficient ideals
one may then conclude $I_{k}=H$.

Hence, the second part follows by observing that
$$
(I_{\mathfrak{p}})_{k}=((I_{\mathfrak{p}_{1}})_{k} \cap R_{\mathfrak{p}})\cap \cdot \cdot \cdot \cap ((I_{\mathfrak{p}_{t}})_{k} \cap R_{\mathfrak{p}})
$$
and
$$
(I_{k})_{\mathfrak{p}}=((I_{\mathfrak{p}_{1}})_{k} \cap R)_{\mathfrak{p}}\cap \cdot \cdot \cdot \cap ((I_{\mathfrak{p}_{t}})_{k} \cap R)_{\mathfrak{p}},
$$
where $\mathfrak{p}_{1},..,\mathfrak{p}_{t} \subseteq \mathfrak{p}$ and $\mathfrak{p}_{t+1},..,\mathfrak{p}_{r} \not\subseteq \mathfrak{p}$
\end{proof}

\begin{thm}\label{height_ass}
Let $(R,\mathfrak{m})$ be a local ring and $I$ an equimultiple ideal.
If $((I^{n})^{*})^{u}=I^{n}$ for all $n$, then $\emph{ht}(P)<s$ for every $P\in \emph{Ass}(G_{I}(R))$.
\end{thm}

\begin{proof}
It is easy to see that the hypothesis implies $(I^{n})^{*}=I^{n}$ and $(I^{n})^{u}=I^{n}$ for all $n$.
Let $P\in \text{Ass}_{\mathcal{R}}(\mathcal{R}/t^{-1}\mathcal{R})$
(where $\mathcal{R}$ is the Extended Rees Algebra) and $\mathfrak{p}=P\cap R$.
Initially we assume $R$ is a domain.
By the Dimension Inequality one has
$$
\text{ht}(P)-\text{ht}(\mathfrak{p}) \leq 1-\text{tr.deg}_{R/\mathfrak{p}}\hspace{0.05cm} \mathcal{R}/P.
$$
We claim that $\text{tr.deg}_{R/\mathfrak{p}}\hspace{0.05cm} \mathcal{R}/P \neq 0$.
Suppose the contrary. Note first that we can assume
$\mathfrak{p}$ is maximal, since $((I_{\mathfrak{p}}^{n})^{*})^{u}=I_{\mathfrak{p}}^{n}$
and $\text{ht}(P)=\text{ht}(P_{\mathfrak{p}})$. Then $\mathcal{R}/P$ is a finitely generated algebra over the field $k:=R/\mathfrak{p}$. Hence,
$\dim \mathcal{R}/P=\text{tr.deg}_{k}\mathcal{R}/P=0$. It implies $\mathcal{R}/P$ is a field, as $\mathcal{R}/P$ is a domain.
Therefore, $R/\mathfrak{p}$ and $ \mathcal{R}/P$ are isomorphic; whence $G_{+} \subseteq P/t^{-1}\mathcal{R}$, which is a contradiction since
$G_{+}$ is a regular ideal.

In conclusion, we can write $P=(t^{-1}\mathcal{R}:at^{r})$ for some homogeneous element $at^{r} \in \mathcal{R} \ \backslash \ t^{-1}\mathcal{R}$.
Hence, $Pat^{r}\in I^{r+1}t^{r}$ for some integer $r\geq 0$ so that $\mathfrak{p}=(I^{r+1}:a)$. This means $\mathfrak{p}\in  \text{Ass}(R/I^{r+1})$. By hypothesis,
$\mathfrak{p}$ is a minimal prime of $I^{r+1}$, so
$\text{ht}(\mathfrak{p})=s$. As $\text{tr.deg}_{R/\mathfrak{p}}\hspace{0.05cm} \mathcal{R}/P\neq 0$, through the above inequality, we obtain
$\text{ht}(P) \leq s$ and therefore $\text{ht}(P/t^{-1}\mathcal{R}) < s$.

The case which $R$ is not a domain is similar. It follows by taking
a minimal prime $Q$ contained in $P$ such that $\text{ht}(P)=\text{ht}(P/Q)$ and after going module a minimal prime.
\end{proof}

\begin{rem}
The equation $((I^{n})^{*})^{u}=I^{n}$ for all $n$ is equivalent to have $(I^{n})^{*}=I^{n}$ for all $n$ and $(I^{n})^{u}=I^{n}$ for all $n$.
Moreover if $I$ is a regular ideal and only $(I^{n})^{u}=I^{n}$ for $n\gg 0$, the condition $(I^{n})^{*}=I^{n}$ for all $n$ implies $(I^{n})^{u}=I^{n}$ for all $n$, since
$\emph{Ass}(R/(I^{n})^{*})\subseteq \emph{Ass}(R/(I^{n+1})^{*})$ for all $n\geq 1$ \emph{(see \cite[p.14]{PZ})}.
\end{rem}

\section{Equimultiple Coefficient Ideals, Associated Graded ring and Serre's Condition $(S_n)$}

In this section, we see necessary and sufficient conditions for Associated graded ring $G_{I}(R)$ to satisfy $S_{1}$ condition. It has a relation to the concept of equimultiple coefficient ideals. Moreover, in Theorem \ref{height_ass} and Theorem \ref{height_ass2}, we generalize the Theorem 4 from \cite{S}, which concerns to the height of associated prime ideals of $G_{I}(R)$. In particular, we obtain that $G_{I}(R)$ satisfies the $S_{1}$ condition
if and only if $(I^{n})_{1}=I^{n}$ for all $n$.

\begin{lem}\label{equiv_height}
Let $(R,\mathfrak{m})$ be a quasi-unmixed local ring and let $I$ be a proper ideal of $R$. Then we have the following:
\begin{itemize}
  \item  [\emph{(1)}] If $\emph{ht}(P)<k$ for every $P\in \emph{Ass}(G_{I}(R))$ then $I_{\mathfrak{p}}^{n}=(I_{\mathfrak{p}}^{n})_{k}$ for all n and every $\mathfrak{p}\in \emph{Min}(R/I)$;
  \item  [\emph{(2)}] Suppose all the powers $I^{n}$ are unmixed ideals, then the converse of $\emph{(1)}$ is valid.
\end{itemize}
\end{lem}

\begin{proof}
To show $\text{(1)}$ let $\mathfrak{p}\in \text{Min}(R/I)$. Then for every
$P_{\mathfrak{p}}\in \text{Ass}_{\mathcal{R}_{\mathfrak{p}}}(\mathcal{R}_{\mathfrak{p}}/t^{-1}\mathcal{R}_{\mathfrak{p}})$ we have
$\text{ht}(P_{\mathfrak{p}})=\text{ht}(P)<k+1$. We use then the result \cite[Theorem 4]{S} to complete the assertion.
For the other assertion, let $P \in \text{Ass}_{\mathcal{R}}(\mathcal{R}/t^{-1}\mathcal{R})$. We have $\mathfrak{p}=P\cap R\in \text{Ass}(R/I^{n})$
which is a minimal on $I^{n}$ by assumption. Therefore, $\text{ht}(P)=\text{ht}(P_{\mathfrak{p}})<k+1$ and the result follows.
\end{proof}

%\label{height_ass2}
\begin{thm}\label{height_ass}
Let $(R,\mathfrak{m})$ be a quasi-unmixed local ring and let $I$ be an equimultiple ideal.
\begin{itemize}
  \item [\emph{(1)}] If $(I^{n})_{k}=I^{n}$ for all $n$, then
$\emph{ht}(P)<k$ for every $P\in \emph{Ass}(G_{I}(R))$.
  \item [\emph{(2)}] If $\emph{ht}(P)<k$ for every $P\in \emph{Ass}(G_{I}(R))$
then $(I^{n})_{k}=(I^{n})^{u}$ for all $n$.
\end{itemize}
In particular, $G_{I}(R)$ satisfies the $S_{1}$ condition
if and only if $(I^{n})_{1}=I^{n}$ for all $n$.
\end{thm}

\begin{proof}
If $(I^{n})_{k}=I^{n}$ we have in particular that $(I^{n})^{u}=I^{n}$ for all $n$. Further by localizing we obtain
$(I^{n}_{\mathfrak{p}})_{k}=((I^{n})_{k})_{\mathfrak{p}}=I^{n}_{\mathfrak{p}}$, by Proposition \ref{prym_deco}.
Now one just applies Lemma \ref{equiv_height} to conclude item $\text{(1)}$.

To show $\text{(2)}$ firstly let $\mathfrak{p}$ be a minimal prime of $I$ and consider the localization $S^{-1}R[It,t^{-1}]$, where $S=R\backslash\mathfrak{p}$.
For each associated prime $S^{-1}P\in \text{Ass}_{S^{-1}R[It,t^{-1}]}\big(\frac{S^{-1}R[It,t^{-1}]}{t^{-1}S^{-1}R[It,t^{-1}]}\big)$,
we have $\text{ht}(S^{-1}P)=\text{ht}(P)<k+1$. Thus, by \cite[Theorem 4]{S}, we obtain, for each $\mathfrak{p}\in \text{Min}(R/I)$, that
$I_{\mathfrak{p}}^{n}=(I_{\mathfrak{p}}^{n})_{k}$ for all $n$.
In particular,
$I_{\mathfrak{p}}^{n}=((I^{n})_{k})_{\mathfrak{p}}$, so that $(I^{n})_{k}=(I^{n})^{u}$ for all positive integer $n$.

Now we consider the case $k=1$.
If $\mathfrak{p}\in \text{Ass}(R/I^{n})$ then there exists a $P\in \text{Ass}_{\mathcal{R}}\frac{\mathcal{R}}{t^{-1}\mathcal{R}}$ (where
$\mathcal{R}$ is the Extended Rees Algebra), such
that $\mathfrak{p}=P\cap R$. Firstly assume $R$ is a domain. By using Dimension Formula we obtain $\text{ht}(\mathfrak{p})=\text{ht}(P)-1+t$, where
$t:=\text{tr.deg}_{\frac{R}{\mathfrak{p}}}\frac{\mathcal{R}}{P}=\text{tr.deg}_{\frac{R_{\mathfrak{p}}}{ \mathfrak{p}R_{\mathfrak{p}  }}}
\frac{\mathcal{R}_{\mathfrak{p}}}{ P_{\mathfrak{p}}}  \leq   \text{tr.deg}_{\frac{R_{\mathfrak{p}}}{ \mathfrak{p}R_{\mathfrak{p}  }}}
\frac{\mathcal{R}_{\mathfrak{p}}}{\mathfrak{p}\mathcal{R}_{\mathfrak{p}}}=s$. Therefore, each associated prime $\mathfrak{p}\in \text{Ass}(R/I^{n})$
is actually a minimal prime of $I^{n}$ for every $n$, as required. The general case may be derived by taking a minimal prime $Q$ of $\mathcal{R}$ such
that $Q \subseteq P$ and $\text{ht}(P)=\text{ht}(P/Q)$. The converse is immediate from $\text{(1)}$.
\end{proof}

\begin{cor}
Let $(R,\mathfrak{m})$ be a quasi-unmixed, analytically unramified
domain satisfying $S_{2}$ condition and let $I$ be an equimultiple ideal such that $\emph{ht}(I)\geq 2$. If $R=\oplus_{n\geq 0}I_{n}t^{n} $ is the
$S2$-ification of the Rees algebra $R[It]$ and
$\emph{ht}(P)<k$ for every $P\in \emph{Ass}(G_{I}(R))$,
then
$$(I^{n})_{k}\subseteq I_{n}.$$
\end{cor}

\begin{proof}
It follows directly from \cite[Proposition 2.10]{C}.
\end{proof}

The Theorem \ref{height_ass} derives the following result, firstly introduced by Noh and Vasconcelos \cite[Theorem 2.5]{NV} for the less general case which
$R$ is a Cohen-Macaulay ring.

\begin{cor}
Let $R$ be a quasi-unmixed ring satisfying $S_{2}$ and $I$ an equimultiple ideal containing a regular element.
If $R[It]$ satisfies $S_{2}$, then all the powers $I^{n}$ are unmixed ideals.
\end{cor}

\begin{proof}
We may assume $R$ is local. It suffices then to use \cite[Theorem 1.5]{BSV} and later apply Theorem \ref{height_ass}.
\end{proof}

\begin{rem}
Grothe, Hermann and Orbanz \emph{\cite[Theorem 4.7]{GHO}} showed that if $I$ is an equimultiple ideal of a Cohen-Macaulay local ring $(R,\mathfrak{m})$,
then the Cohen-Macaulayness of $G_{I}(R)$ implies
the Cohen-Macaulayness of $G_{I^{m}}(R)$, for all $m\geq 1$. Also Shah \emph{\cite[Corollary 5(C)]{S}} showed the same result
when $R$ is just quasi-unmixed but $I$ an
$\mathfrak{m}$-primary ideal.

Below we see that a similar result for the $S_{1}$ condition can be obtained immediately through coefficient ideals.
\end{rem}

\begin{cor}\label{S1property}
Let $(R,\mathfrak{m})$ be quasi-unmixed local ring and $I$ an equimultiple ideal.
If $G_{I}(R)$ satisfies $S_{1}$, then so does $G_{I^{m}}(R)$, for all $m\geq 1$.
\end{cor}

Due to the above result and \cite[Theorem 1.5]{BSV}, we obtain the following

\begin{cor}\label{cor_S2}
Let $(R,\mathfrak{m})$ be a quasi-unmixed local ring satisfying $S_{2}$ and $I$ an equimultiple ideal containing a regular element.
If $R[It]$ satisfies $S_{2}$, then so does $R[I^{m}t]$, for all $m\geq 1$.
\end{cor}

\begin{thm}\label{grade}
Let $(R,\mathfrak{m})$ be a quasi-unmixed local ring and $I$ an equimultiple ideal of analytic spread $s$.
If $\emph{depth}\hspace{0.05cm}G_{I}(R)\geq k$, where
$1\leq k \leq s$,
then $(I^{n})_{j}=(I^{n})^{u}$ for all $n$ and $s+1-k\leq j \leq s$.
\end{thm}

\begin{proof}
By using the fact that $G\otimes_{R}R_{\mathfrak{p}}$ is flat over $G$, one can conclude that
$\text{depth}\hspace{0.05cm}G_{I}(R)\geq k$ implies $\text{depth}\hspace{0.05cm}G_{I_{\mathfrak{p}}}(R_{\mathfrak{p}})\geq k$
for each $\mathfrak{p}$ prime. So by \cite[Theorem 5]{S}, we have $I_{\mathfrak{p}}^{n}=(I_{\mathfrak{p}}^{n})_{j}$ for all $n$
and each minimal prime $\mathfrak{p}$ of $I$. Hence, $I_{\mathfrak{p}}^{n}=((I^{n})_{j})_{\mathfrak{p}}$ and therefore
$(I^{n})_{j}=(I^{n})^{u}$ for all positive integer $n$, as all coefficients ideals are unmixed ideals.
\end{proof}

\begin{rem}\label{rem_grade}
As it can be seen in the above proof, if $I$ is an arbitrary equimultiple ideal and $\emph{depth}\hspace{0.05cm}G_{I}(R)_{+}\geq k$ one has
$I_{\mathfrak{p}}^{n}=(I_{\mathfrak{p}}^{n})_{j}$ for all $n$, $s+1-k\leq j \leq s$
and each minimal prime $\mathfrak{p}$ of $I$.
\end{rem}

%\begin{prop}\label{prop_grade}
%Let $(R,\mathfrak{m})$ be a quasi-unmixed local ring and $I$ a proper ideal of analytic spread $s$. If $\emph{grade}\hspace{0.05cm}G_{I}(R)_{+}\geq s$
%and some power $I^{i}$ is unmixed then $G_{I}(R)$ satisfies the $S_{1}$ condition.
%\end{prop}

%\begin{proof}
%Since $\text{ht}(G_{+})=\text{ht}(I)$, the ideal $I$ is equimultiple. Let $x_{1},...,x_{s}$ be a minimal reduction of $I$ such that
%$x_{1}',...,x_{s}'\in G_{I}(R)$ is a superficial sequence. By assumption we can obtain that $x_{1},...,x_{s}$ is a regular sequence.
%Hence, $I$ can be generated by a regular sequence, thus there is an isomorphism of graded rings
%$$
%A=(R/I)[X_{1},..,X_{s}]\cong G_{I}(R),
%$$
%where $A$ is a polynomial ring with coefficients in $R/I$. We can then conclude that
%$\text{Ass}_{R}(I^{i}/I^{i+1})=\text{Ass}_{R}(R/I)$ for each $i$. By using the exact sequence
%$$
%0 \rightarrow I^{i}/I^{i+1} \rightarrow R/I^{i+1} \rightarrow R/I^{i} \rightarrow 0,
%$$
%it follows by induction that $\text{Ass}_{R}(R/I^{n}) = \text{Ass}_{R}(R/I)$ for $n\geq 1$.
%If then some power $I^{i}$ is unmixed, we conclude that $I^{n}=(I^{n})^{u}$. The result follows by
%Theorem \ref{height_ass}.
%\end{proof}

\begin{prop}\label{complet_inter}
Let $(R,\mathfrak{m})$ be a quasi-unmixed local ring satisfying the $S_{s+1}$ condition and $I$ an ideal with analytic spread $s$ such that
$\emph{grade}\hspace{0.05cm}I=s(I)$.
Suppose $s(I_{\mathfrak{p}})=\mu(I_{\mathfrak{p}})$ for every $\mathfrak{p} \in \emph{Min}(R/I)$.
Then $G_{I}(R)$ satisfies $S_{1}$.
\end{prop}

\begin{proof}
By hypothesis, there exists a minimal reduction $J$ of $I$ generated by
a regular sequence of length $s$ and $J_{\mathfrak{p}}=I_{\mathfrak{p}}$ for each $\mathfrak{p}\in \text{Min}(R/I)$ since $I_{\mathfrak{p}}$
has no proper reduction. Once $R$ satisfies the $S_{s+1}$ condition we have $J$ is unmixed. Hence, $I=J$ is generated by a regular sequence of length $s$. In particular the generating set of $I$ form a quasi-regular sequence, thus
$\text{grade}\hspace{0.05cm} G_{I}(R)_{+}\geq s$. Moreover there is an isomorphism of graded rings
$$
A=(R/I)[X_{1},..,X_{s}]\cong G_{I}(R),
$$
where $A$ is a polynomial ring with coefficients in $R/I$. We can then conclude that
$\text{Ass}_{R}(I^{i}/I^{i+1})=\text{Ass}_{R}(R/I)$ for each $i$. By using the exact sequence
$$
0 \rightarrow I^{i}/I^{i+1} \rightarrow R/I^{i+1} \rightarrow R/I^{i} \rightarrow 0,
$$
it follows by induction that $\text{Ass}_{R}(R/I^{n}) = \text{Ass}_{R}(R/I)$ for $n\geq 1$.
Since $I$ is unmixed, we conclude that $I^{n}=(I^{n})^{u}$ for all $n$. The result follows then by
Theorem \ref{grade} and Theorem \ref{height_ass}.
\end{proof}

An ideal $I$ is a \emph{locally complete intersection} if $\text{ht}(I_{\mathfrak{p}})=\mu(I_{\mathfrak{p}})$ for each $\mathfrak{p}\in \text{Ass}(R/I)$.

\begin{cor}
Let $(R,\mathfrak{m})$ be a quasi-unmixed local ring satisfying the $S_{s+1}$ condition and $I$ an ideal with analytic spread $s$ such that
$\emph{grade}\hspace{0.05cm}I=s(I)$.
Suppose $I$ is a locally complete intersection.
Then $G_{I}(R)$ satisfies $S_{1}$.
\end{cor}

\begin{rem}
In the setup of Proposition \ref{complet_inter}, we obtain $(I^{n})_{s}= \cdot \cdot \cdot =(I^{n})_{1}=I^{n}$ for all positive integer $n$.
\end{rem}

\begin{defi}
Let $R$ be a local ring and $I$ a proper ideal of $R$.
The reduction number $r(I)$ of $I$ is defined to be
$$
r(I)=\emph{min} \{ n \ | \ \text{there exists a minimal reduction } J \text{ of } I \text{such that } I^{n+1} =JI^{n} \}.
$$
\end{defi}

\begin{prop}
Let $(R,\mathfrak{m})$ be a quasi-unmixed local ring satisfying the $S_{s+1}$ condition and $I$ an ideal such that $\emph{grade}\hspace{0.05cm}I=s(I)$.
Suppose some power $I^{i}$ with $r(I^{i})\leq 1$ is an unmixed ideal.
If $\emph{grade}\hspace{0.05cm}G_{I}(R)_{+}\geq k$, where $1\leq k\leq s$,
then $(I^{n})_{j}=I^{n}$ for all $n$ and $s+1-k\leq j \leq s$;
\end{prop}

\begin{proof}
Because of Theorem \ref{grade}, it suffices to show $(I^{n})^{u} =I^{n}$ for all $n$.
The hypothesis $\text{grade}\hspace{0.05cm}I=s(I)$ gives the existence of a minimal reduction $J$ of $I$ generated by
a regular sequence of length $s$. Since $R$ satisfies $S_{s+1}$, the ideal $J$ is unmixed.
Further, by \cite[(1.2)]{HLS}, we have $(I^{n})^{*}=I^{n}$ for all $n$.
Hence, $\text{Ass}(R/I^{n})\subseteq \text{Ass}(R/I^{n+1})$ for all $n$, by \cite[Lemma 6.6]{PZ}.
By hypothesis we then get $\text{Ass}(R/I^{n})=\text{Min}(R/I^{n})$ for all $n\leq i$.
By assumption on $r(I^{i})$, there exists a minimal reduction $J$ of $I^{i}$ such that $JI^{i}=(I^{i})^{2}$.
Now consider the exact sequence
$$
0\rightarrow J/JI^{i} \rightarrow R/JI^{i} \rightarrow R/J \rightarrow 0,
$$
where $J/JI^{i}\simeq (R/I^{i})^{s}$.
Since $J$ is an unmixed ideal one may then conclude that $(I^{i})^{2}$ is unmixed.
By considering the following
exact sequence
$$
0\rightarrow J^{2}/J^{2}I^{i} \rightarrow R/(I^{i})^{3} \rightarrow R/J^{2} \rightarrow 0,
$$
we get $(I^{i})^{3}$ is unmixed. We have then obtained that $I^{n}$ is unmixed for infinitely many $n$.
Therefore, all the powers $I^{n}$ are unmixed ideals.
\end{proof}

\begin{lem}\label{lem_alg_str}
Let $(R,\mathfrak{m})$ be a quasi-unmixed local ring of infinite residue field and $I$ an equimultiple ideal. For all $N \geq 1$ and all reduction
$x=x_{1},...,x_{t}$ of $I^{N}$,
we have
$$
(I^{N+1}:x_{1},...,x_{k})\subseteq I_{k}, \ for \ 1\leq k\leq d.
$$
\end{lem}

\begin{proof}
It is easy to see we may assume $x$ is a minimal reduction. Fix any $N \geq 1$ and let $x_{1},...,x_{s}$ be a minimal reduction of $I^{N}$.
By \cite[Theorem 2]{S}, we have $(I^{N+1}:x_{1},...,x_{k})_{\mathfrak{p}}\subseteq (I_{\mathfrak{p}})_{k}$ for each $\mathfrak{p}\in \text{Min}(R/I)$.
Hence, for each $\mathfrak{p}\in \text{Min}(R/I)$, the equality $e_{i}((I^{N+1}:x_{1},...,x_{k})_{\mathfrak{p}})=e_{i}(I_{\mathfrak{p}})$ is true for $0\leq i \leq k$.
The result then follows by maximality of $I_{k}$.
\end{proof}

\begin{lem}\label{syst_equi}
Let $(R,\mathfrak{m})$ be quasi-unmixed ring and $I$ an equimultiple ideal with analytic spread $s$. Let $x_{1},...,x_{s}\in I^{N}$ for some $N\geq 1$
and $x'_{1},...,x'_{s}$ their images in $I^{N}/I^{N+1}$. Let $a_{1},...,a_{t}\in R$ be a system of parameters module $I$.
Then
$$
\begin{array}{c}
  a'_{1},...,a'_{t},x'_{1},...,x'_{s} \text{ is a system of parameters of }G_{I}(R)
 \vspace{0.2cm}\\
\text{ if and only if }  x_{1},...,x_{s} \text{ form a minimal reduction of } I^{N}
\end{array}
$$
\end{lem}

\begin{proof}
Denote $G_{I}(R)$ by $G$ and the ideal $(a'_{1},...,a'_{t})G_{I}(R)$ by $L$.
If the sequence $a'_{1},...,a'_{t},x'_{1},...,x'_{s}$ is a system of parameters of $G$,
we have $\sqrt{(G/L)_{+}}=(x'_{1},...,x'_{s})G/L$. Then $G_{+}^{n}\subseteq (x_{1},...,x_{s})G+L$ for some positive integer $n$. Thus, it is easy to see that
$G_{+}^{n}\subseteq (x_{1},...,x_{s})G+LG_{+}^{n}$. By Nakayama's Lemma, one can conclude $G_{+}^{n}\subseteq (x_{1},...,x_{s})G$. Therefore,
$x_{1},...,x_{s}$ form a minimal reduction of $I^{N}$.
The converse follows analogously to \cite[Corollary 2.7]{GHO}.
\end{proof}

The next result is also a generalization of Theorem 4 in \cite{S} as it can be easily observed.

\begin{thm}\label{height_ass2}
Let $(R,\mathfrak{m})$ be a quasi-unmixed local ring and $I$ an equimultiple ideal. Let $a_{1},...,a_{r}$ be a system of parameters modulo $I$
and let $L=(a'_{1},...,a'_{r})$ denote the ideal generated by their images in $G_{I}(R)$.
Fix $k\in\{1,..,s\}$. If $I^{n}=(I^{n})_{k}$ for all $n$
then $\emph{ht}(P+L)<k+d-s$ for every $P\in \emph{Ass}(G_{I}(R))$.
\end{thm}

\begin{proof}
Let $P$ be an associated prime of $G_{I}(R)$. Since $P$ is graded we have $P=(0':y')$ where $y'$ is the image in $G_{I}(R)$
of $y\in I^{n-1}-I^{n}$ for some $n\geq 1$. Suppose $\text{ht}(P+L) \geq k+d-s$. Then we obtain $\dim \frac{G/L}{P(G/L)}\leq s-k$.
By \cite[Lemma 2(E)]{S} we can get a homogeneous system of parameters (of equal degree) $\overline{x_{1}},...,\overline{x_{s}}$ for $G/L$ such that
$\overline{x_{1}},...,\overline{x_{k}}\in P(G/L)$. Let $x'_{1},...,x'_{s}\in G_{I}(R)$ be homogeneous inverse images of
$\overline{x_{1}},...,\overline{x_{s}}$ respectively so that $x'_{1},...,x'_{k}\in P$.
Since $a'_{1},...,a'_{t},x'_{1},...,x'_{s}$ is a system of parameters of $G_{I}(R)$ (\cite[Proposition 2.6]{GHO}) we can then use
Lemma \ref{syst_equi} to obtain that
$x_{1},...,x_{s}$ is a minimal reduction of $I^{m}$ for
some $m$. It is easy to see we may assume $m=nN$. Hence, $y(x_{1},...,x_{k}) \subseteq I^{n-1+nN+1}$, and so $y\in ((I^{n})^{N+1}: x_{1},...,x_{k})$.
By Lemma \ref{lem_alg_str} one has $y\in (I^{n})_{k}=I^{n}$, which is a contradiction.
\end{proof}

\begin{defi}
Let $R$ be a local ring and $I$ a proper ideal of $R$.
If for each $\mathfrak{p}\in \emph{Min}(R/I)$, the localization $I_{\mathfrak{p}}$
has reduction number $r(I_{\mathfrak{p}})\leq t$, it is said
that $I$ has generically reduction number $t$.
\end{defi}

\begin{prop}\label{gene}
Let $(R,\mathfrak{m})$ be a Cohen-Macaulay local ring of infinite residue field and $I$ an equimultiple ideal of analytic spread $s$.
Assume $R/I$ satisfies the $S_{1}$ condition and $I$ has generically reduction number 1. Then $G_{I}(R)$ satisfies the $S_{1}$ condition.
\end{prop}

\begin{proof}
Let $J$ be a minimal reduction of $I$ and consider the exact sequence
$$
0\rightarrow J/JI \rightarrow R/JI \rightarrow R/J \rightarrow 0,
$$
where $J/JI\simeq (R/I)^{s}$. Hence, since $R$ is Cohen-Macaulay
and because $I$ is unmixed, one may conclude $JI$ is unmixed.
In this way, $I^{2}R_{\mathfrak{p}}= JIR_{\mathfrak{p}}$ for all $\mathfrak{p} \in \text{Ass}(R/JI)=\text{Min}(R/JI)$, and then
$I^{2}=JI$. Consider now the following
exact sequence
$$
0\rightarrow J^{2}/J^{2}I \rightarrow R/I^{3} \rightarrow R/J^{2} \rightarrow 0.
$$
Since the associated prime ideals of $J^{n}$ are the same as those of $J$ and
$J^{2}/J^{2}I$ is isomorphic to a power of $R/I$, one may conclude
that $I^{3}$ is unmixed. Inductively we obtain that
all powers $I^{n}$ are unmixed ideals.
On the other hand, if $\mathfrak{p}\in \text{Ass}_{R}(R/I)$
we have $R_{\mathfrak{p}}/I_{\mathfrak{p}}$ is a Cohen-Macaulay ring and hence, by \cite[Proposition 26.12]{HIO} we obtain that
$G_{I_{\mathfrak{p}}}(R_{\mathfrak{p}})$ is Cohen-Macaulay for all minimal prime $\mathfrak{p}$
of $I$, and therefore we may use \cite[Theorem 4]{S} to obtain
$I_{\mathfrak{p}}^{n}=(I_{\mathfrak{p}}^{n})_{1}$ for all $n$ and all $\mathfrak{p}\in \text{Min}(R/I)$.
Now one may just apply Lemma \ref{equiv_height} to conclude
the proof.
\end{proof}

In \cite{CP}, Corso and Polini indicated a method to provide ideals $I$ of reduction number 1. In this way, through above proposition we may produce
associated graded rings $G_{I}(R)$ satisfying the $S_{1}$ condition.

\begin{cor}\label{construct_of_G_I(R)}
Let $(R,\mathfrak{m})$ be a Cohen-Macaulay ring, $\mathfrak{p}$ a prime ideal of height $g$ such that $R_{\mathfrak{p}}$ is not a regular local ring
and $J=(x_{1},...,x_{g}) \subseteq \mathfrak{p}$ a regular
sequence. Set $I=J:\mathfrak{p}$. Then $G_{I}(R)$ satisfies $S_{1}$.
\end{cor}

\begin{proof}
Since $\text{Ass}(R/I)\subseteq \text{Ass}(R/J)$, the ideal $I$ is unmixed.
The result follows then from \cite[Theorem 2.3]{CP}.
\end{proof}

\begin{agra}
Pedro Lima thanks Sathya Sai Baba for the guidance. Both authors are very grateful to Professor Daniel Katz for his encouragement, advices and reviews. Pedro Lima thanks the University of Kansas for all the facilities.
\end{agra}

%Let $R$ be a Cohen-Macaulay local ring of infinite residue field and $I$ a proper ideal having generically reduction exponent 2.
%Grothe, Herrmann and Orbanz showed in \cite{GHO} that if $R/I^{n}$ is Cohen-Macaulay for all $n$ then so does $G_{I}(R)$.
%The following result gives an analogous for the $S_{1}$ condition.

%\begin{prop}
%Let $(R,\mathfrak{m})$ be a Cohen-Macaulay local ring of infinite residue field and $I$ a proper ideal.
%Assume $R/I^{n}$ satisfies the $S_{1}$ condition for all $n$ and $I$ has generically reduction exponent 2. Then $G_{I}(R)$ satisfies the $S_{1}$ condition.
%\end{prop}

%\begin{proof}
%In proof of Proposition \ref{gene} we saw that if $R/I^{n}$ satisfies $S_{1}$ then $R_{\mathfrak{p}}/I_{\mathfrak{p}}^{n}$ is
%Cohen-Macaulay for all $\mathfrak{p}\in \text{Ass}(R/I^{n})$ and $I^{n}$ is unmixed. By \cite[Proposition 5.9]{GHO} one has
%$G_{I_{\mathfrak{p}}}(R_{\mathfrak{p}})$ is Cohen-Macaulay for each $\mathfrak{p}\in \text{Min}(R/I)$. The result follows then by Lemma \ref{equiv_height}.
%\end{proof}

\textbf{Department of Mathematics, Institute of Mathematics and Computer Science, ICMC, University of S\~{a}o Paulo, BRAZIL.
}

\emph{E-mail address}: apoliano27@gmail.com

\hspace{1cm}

\textbf{Department of Mathematics, Institute of Mathematics and Computer Science, ICMC, University of S\~{a}o Paulo, BRAZIL.
}

\emph{E-mail address}: vhjperez@icmc.usp.br

\hspace{1cm}
\end{document}